\documentclass[smallextended,envcountsect]{svjour3}

\smartqed

\usepackage{graphicx}

\usepackage{latexsym,amsmath,amsfonts,amscd,amssymb,verbatim,hyperref}
\usepackage{color}
\usepackage{cite}
\usepackage[shortlabels]{enumitem}

\def\tto{\;{\lower 1pt \hbox{$\rightarrow$}}\kern -10pt
	\hbox{\raise 2pt \hbox{$\rightarrow$}}\;}

\def\epsilon{\varepsilon}

\def\R{\Bbb R}

\def\ox{\bar{x}}

\def\dim{\mbox{\rm dim}\,}

\def\sint{\mbox{\rm int}\,}

\begin{document}
	
	\title{The Hartman-Stampacchia Theorem and the Maximum Displacements of Nonvanishing Continuous Vector-Valued Functions}
	
	\author{Nguyen Nang Thieu \and Nguyen Dong Yen}
	
	\institute{Nguyen Nang Thieu \at Institute of Mathematics, Vietnam Academy of	Science and Technology\\
		18 Hoang Quoc Viet, Hanoi 10307, Vietnam\\
		nnthieu@math.ac.vn\and 
		Nguyen Dong Yen, Corresponding author \at
		Institute of Mathematics, Vietnam Academy of	Science and Technology\\
		18 Hoang Quoc Viet, Hanoi 10307, Vietnam\\
		ndyen@math.ac.vn		
	}
	
	\date{Received: date / Accepted: date}
	
	\titlerunning{Maximum Displacements of Nonvanishing Continuous Vector-Valued Functions}
	
	\authorrunning{N. N. Thieu, N. D. Yen}
	
	\maketitle
	

\begin{abstract} This paper aims at giving solutions to six interesting interconnected open questions suggested by Professor Biagio Ricceri. The questions focus on the behavior of nonvanishing continuous vector-valued functions in finite-dimensional normed spaces as well as in infinite-dimensional normed spaces. Using the celebrated Hartman-Stampacchia Theorem (1966) on the solution existence of variational inequalities, we establish sharp lower estimates for the maximum displacements of nonvanishing continuous vector-valued functions. Then, combining the obtained results with suitable tools from functional analysis and several novel geometrical constructions, we get the above-mentioned solutions. 
	
\end{abstract}


\section{Introduction}\label{Sect_Intro}\label{Sect-1}

Through a series of seminar discussions during our visit to Catania University in October 2024, Professor Biagio Ricceri~\cite{Ricceri_2024} formulated  
six interesting interconnected open questions on the behavior of nonvanishing continuous vector-valued functions in finite-dimensional normed spaces as well as in infinite-dimensional normed spaces. Roughly speaking, these questions revolve around the maximum displacements and generalized maximum displacements of continuous vector-valued functions. The aim of the present paper is to obtain complete or partial solutions to the questions. It turns out that the Hartman-Stampacchia Theorem on the solution existence of variational inequalities (see \cite[Lemma~3.1]{HS_1966} and~\cite[Theorem~3.1]{KS_1980}) serves as a crucial key, unlocking a path to these solutions, which also rely on appropriate tools from functional analysis and several novel geometrical constructions. The obtained results enrich our knowledge on  the behavior of nonvanishing continuous vector-valued functions in finite-dimensional normed spaces as well as in infinite-dimensional normed spaces.

\medskip
Before formulating the notions of maximum displacements and generalized maximum displacements of continuous vector-valued functions, we need to explain some standard notations.

\medskip
Throughout this paper, let $E$ be a real normed space equipped with the norm $\|\cdot\|$ and $K$ be a nonempty convex subset of $E$. The dual space of $E$ is denoted by $E^*$. The norm in $E^*$ is denoted by $\|\cdot\|_{E^*}$. The boundary of a set $D\subset E$ is abbreviated to $\partial D$, while  the interior of $D$ is denoted by $\mbox{\rm int}\,D$.  The open ball (resp., closed ball)  in $E$ with center~$x$ and radius $r$ is denoted by $B_{E} (x, r) $ (resp., $\bar{B}_{E} (x, r) $). The linear subspace generated by some vectors $x^1,\dots,x^m$ in $E$ is abbreviated to ${\rm span}\{x^1,\dots,x^m\}$. From now on, let ${\cal H}$ be a real Hilbert space equipped with the scalar product $\langle \cdot,\cdot \rangle$ and the norm $ \Vert \cdot \Vert $.

\begin{definition} (\textbf{Maximum displacement})\label{def1} The \textit{maximum displacement} of a continuous function $f:K\to E$ is the quantity $\sup\limits_{x\in K}\|f(x)-x\|$.
\end{definition}

If $x\in K$ and $f(x)= x$, that is~ $\|f(x) - x\|=0$, then $x$ is a fixed point of~$f$. The function $f$ needs not to have fixed points. The norm $\|f(x) - x\|$ shows \textit{how far $f$ `displaces' a point $x\in K$ away from its original position}. For this reason, the quantity $\sup\limits_{x\in K}\|f(x)-x\|$ is called the maximum displacement of $f$. A counterpart of the latter is the number $\inf\limits_{x\in K}\|f(x)-x\|$, which signifies the deviation of $f$ from having a fixed point, has been studied by~Goebel in~\cite{Goebel_1973} under the name the \textit{minimal displacement}.

\begin{definition} (\textbf{Generalized maximum displacement})\label{def2}  The \textit{generalized maximum displacement} of a pair of continuous functions $(\varphi,\psi)$, where $\varphi:K\to E^*$ and $\psi:K\to E^*$ with $K$ being a nonempty convex subset of $E$, is the quantity $\sup\limits_{x\in K}\|\varphi(x)-\psi(x)\|_{E^*}$.
\end{definition}

 If $E$ is a Hilbert space, then by the Riesz-Fr\'echet Representation Theorem~\cite[Theorem~5.5]{Brezis_2011} we can identify the dual space $E^*$ with $E$. In that case, setting $\varphi(x)=x$ for all $x\in K$ and using the notations of Definition~\ref{def2}, we see at once that the generalized maximum displacement of the pair $(\varphi,\psi)$ coincides with the maximum displacement of a continuous function $\varphi$ introduced by Definition~\ref{def1}. Thus, in the Hilbert space setting, Definition~\ref{def2} is a natural extension of Definition~\ref{def1}.

\medskip
The above-mentioned six questions from~\cite{Ricceri_2024} can be divided into three groups:

\begin{itemize}
	\item[(a)] Questions 1, 3, and 4 are about maximum displacements.
	\item[(b)] Question 2 is about the approximate eigenvalues of continuous functions, in a sense.
	\item[(c)] Questions 5 and 6 are about generalized maximum displacements.
	\end{itemize}

Question~1 asks to find a continuous function $f:\bar B_{\mathbb R^n}(0,1) \to \R^n$ such that $f(x) \neq 0$ for all $x \in \bar B_{\mathbb R^n}(0,1)$ and the maximum displacement of $f$, denoted by $d_f$, is strictly smaller than a real number $\gamma_f$ (see formula~\eqref{gf} below). 

\medskip
In Section~\ref{Sect-2}, thanks to the Hartman-Stampacchia Theorem  and a valuable observation of Professor Ricceri, we will be able to prove in Theorem~\ref{thm_1a} that $d_f\geq \gamma_f$. This result not only solves the question in the negative, but also establishes a lower bound for the maximum displacement under consideration. Through concrete examples, we will show that the inequality $d_f\geq \gamma_f$ may not hold if the nonvanishing property of~$f$ is violated. However, the inequality can still hold and, in some cases, it even holds as an equality, for many continuous functions $f:\bar B_{\mathbb R^n}(0,1) \to \R^n$ with $f(u)= 0$ for some $u\in \bar B_{\mathbb R^n}(0,1)$.

\medskip
Section~\ref{Sect-3}  extends Theorem~\ref{thm_1a} to the case where instead of the closed ball one considers an arbitrary nonempty compact convex subset of~$\R^n$ containing~$0$ as an interior point. In the same section, the problem of finding sharp estimates for the maximum displacements and the effect of changing the norm in $\mathbb R^n$ are studied in detail. The obtained theorems and corollaries are crucial for our subsequent investigations.   

\medskip
Thanks to the results of Sections~\ref{Sect-2} and~\ref{Sect-3}, in  Section~\ref{Sect-4} we are able to derive complete solutions for the above Questions~2--5 and a partial solution for Question~6.

\medskip
Some concluding remarks are given in Section~\ref{Sect-5}.

\medskip
Despite the fact that the maximum displacement has been used in many papers in mechanics and physics, both notions described by Definitions~\ref{def1} and~\ref{def2} seem to be new in mathematical research. In any case, to the best of our knowledge, all the results of this paper and the ways to obtain them are completely new.

\section{The First Question and Its Solution}\label{Sect-2}

The first question raised by Professor Biagio Ricceri~\cite{Ricceri_2024} is stated as follows.

\medskip
\noindent {\bf Question 1:} {\it Whether there exists a continuous function $f:\bar B_{\mathbb R^n}(0,1) \to \R^n$ such that $f(x) \neq 0$ for all $x \in \bar B_{\mathbb R^n}(0,1)$ and the strict inequality
\begin{equation}\label{ine1}
		\sup_{x \in \bar B_{\mathbb R^n}(0,1)} \|f(x) - x\| < 1 + \inf_{x \in \bar B_{\mathbb R^n}(0,1)} \|f(x)\|
\end{equation} holds true?}

\medskip
Specializing Definition~\ref{def1} to the case where $f$ is a continuous function from $\bar B_{\mathbb R^n}(0,1)$ to $\R^n$, we see that the maximum displacement of $f$ is the number
\begin{equation}\label{df}
	d_f:=\sup\limits_{x \in \bar B_{\mathbb R^n}(0,1)} \|f(x) - x\|. 
\end{equation}
This notation allows us to restate Question~1 as follows: \textit{Is there any nonvanishing continuous function $f:\bar B_{\mathbb R^n}(0,1) \to \R^n$ such that the number 
\begin{equation}\label{gf}
\gamma_f:=1 + \inf\limits_{x \in \bar B_{\mathbb R^n}(0,1)} \|f(x)\|
\end{equation}
is strictly larger than the maximum displacement~$d_f$?}

\medskip
To solve the above question, we will use a valuable observation of Professor Ricceri and the next fundamental theorem on the solution existence of finite-dimensional variational inequalities.
 
\begin{theorem}{\rm (\textbf{The Hartman-Stampacchia Theorem}; See~\cite[Lemma~3.1]{HS_1966} and~\cite[Theorem~3.1]{KS_1980})}\label{hs_thm}
Let $K\subset \R^n$ be a nonempty compact convex set and let $f:K\to\R^n$ be a continuous vector-valued function. Then, there exists a point $\bar x\in K$ such that
\begin{equation}\label{Thm_H-S}
	\langle f(\bar x), y-\bar x\rangle \geq 0\quad\mbox{\rm for all}\ y\in K.
\end{equation}
\end{theorem}

The problem of finding a point $\bar x\in K$ satisfying condition~\eqref{Thm_H-S} is called the \textit{variational inequality} defined by the constraint set $K$ and the operator~$f$, which defines a vector field on $K$ by assigning to every point $x\in K$ the vector~$f(x)$.

\medskip
The following theorem answers Question~1 in the negative.  Namely, there is no continuous vector-valued function with the desired properties.  

\begin{theorem}\label{thm_1a}
	For any nonvanishing continuous function $f:\bar B_{\mathbb R^n}(0,1)\to\R^n$, one has 
	\begin{equation}\label{inq1}
	d_f \geq \gamma_f,
	\end{equation} where $d_f$ and $\gamma_f$ are defined respectively by~\eqref{df} and~\eqref{gf}.
\end{theorem}
\begin{proof} Suppose to the contrary that there exists a nonvanishing continuous function $f:\bar B_{\mathbb R^n}(0,1)\to\R^n$ such that~\eqref{inq1} fails, i.e., the strict inequality~\eqref{ine1} is valid. 
Then, for any $x\in \bar B_{\mathbb R^n}(0,1)$ one has
	$$\|f(x)-x\| \leq \sup_{z\in\bar B_{\mathbb R^n}(0,1)}\|f(z)-z\| < 1+ \inf_{z\in\bar B_{\mathbb R^n}(0,1)} \|f(z)\|\leq  1+  \|f(x)\|.$$
This yields
	$$\|f(x)-x\|^2 < (1+  \|f(x)\|)^2\quad\mbox{\rm for any}\ 	x\in \bar B_{\mathbb R^n}(0,1)$$
	or, equivalently,
	$$- 2\langle f(x),x\rangle +\|x\|^2 < 1+2\|f(x)\| \quad\mbox{\rm for any}\ 	x\in \bar B_{\mathbb R^n}(0,1).$$
Therefore, we have for all $x\in\partial \bar B_{\mathbb R^n}(0,1)$ that
\begin{equation}\label{inq2}
- \langle f(x),x\rangle < \|f(x)\|.
\end{equation} Thus, the fulfillment of the strict inequality~\eqref{inq2} for every $x\in\partial \bar B_{\mathbb R^n}(0,1)$ is \textit{a necessary condition} for having ~\eqref{ine1}\footnote{This important observation was shown to us by Professor Biagio Ricceri.}.

Now, applying Theorem~\ref{hs_thm} to the variational inequality defined by the nonempty compact convex constraint set $K:=\bar B_{\mathbb R^n}(0,1)$ and the operator $f$, we can find a vector $\ox\in  \bar B_{\mathbb R^n}(0,1)$ such that
\begin{equation}\label{VI_f_1}
	\langle f(\bar x), y-\bar x\rangle \geq 0\quad \mbox{\rm for all}\ y\in \bar B_{\mathbb R^n}(0,1).
\end{equation}
If $\ox \in B_{\mathbb R^n}(0,1)={\rm int}\, K$, then from~\eqref{VI_f_1} we can easily deduce that $f(x) = 0$, which contradicts the condition saying that $f$ is nonvanishing on $K$. So, one must have $\ox\in\partial \bar B_{\mathbb R^n}(0,1)$. Combining this with~\eqref{VI_f_1}, we get the representation  $f(\bar x)=-\lambda \bar x$ for some $\lambda\geq 0$. The situation where $\lambda=0$ is excluded, as $f$ is nonvanishing on $K$. Hence, $f(\bar x)=-\lambda \bar x$ with $\lambda>0$. Therefore, we have
$$-\langle f(\ox ),\ox \rangle = \lambda =\| f(\ox)\|,$$
which is a contradiction to~\eqref{inq2}. 

We have thus shown that~\eqref{inq1} holds and completed the proof. $\hfill\Box$
\end{proof}

\begin{remark} {\rm The assumption on the nonvanishing property of  $f$ cannot be dropped in the formulation of Theorem~\ref{thm_1a}. To see that the inequality~\eqref{inq1} may not hold if the assumption is violated, it suffices to choose $f(x)=x$ for all $\bar B_{\mathbb R^n}(0,1)$ with $n\geq 1$ and observe that $d_f=0$, while $\gamma_f=1$.}
\end{remark}

\begin{remark} {\rm For many vanishing functions, the inequality~\eqref{inq1} may hold and, in some cases, it even holds as an equality. As an example, let us  choose $n=2$, $$f_\alpha(x)=\big(x_1\cos\alpha-x_2\sin\alpha,\, x_1\sin\alpha +x_2\cos\alpha\big)$$ for all $x=(x_1,x_2)\in\bar B_{\mathbb R^2}(0,1)$, where $\alpha\in [-\pi,\pi]$ is a parameter. Then, $f_\alpha:\bar B_{\mathbb R^2}(0,1)\to \bar B_{\mathbb R^2}(0,1)$ is the rotation around 0 of angle $\alpha$. Clearly, $\gamma_{f_\alpha}=1$ for all $\alpha\in [-\pi,\pi]$. 
		
For $\alpha=\pi/3$ and $\alpha=-\pi/3$, it is easily verified that $d_{f_\alpha}=\gamma_{f_\alpha}$. So, with $f=f_\alpha$,~\eqref{inq1} holds as an equality. 

For every $\alpha\in [-\pi,\pi]\setminus [-\pi/3,\pi/3]$, one has $d_{f_\alpha}>\gamma_{f_\alpha}$. Thus,~\eqref{inq1} holds as a strict inequality for $f=f_\alpha$. 

For every $\alpha\in (-\pi/3,\pi/3)$, one has $d_{f_\alpha}<\gamma_{f_\alpha}$. This means that the inequality~\eqref{inq1} fails to hold for $f=f_\alpha$.}
\end{remark}

The inequality~\eqref{inq1} tells us that the number $\gamma_f$, which can be computed easily in many situations, is a lower estimate for $d_f$. Various extensions and the sharpness of the estimate~\eqref{inq1} will be the subject of our considerations in the forthcoming section.

\section{Lower Estimates for the Maximum Displacements}\label{Sect-3}

 \subsection{Lower estimates for functions defined on a class of convex subsets of~$\R^n$}
 
As Theorem~\ref{hs_thm} can be applied to any nonempty compact convex subset of $\R^n$, it is of interest to extend Theorem~\ref{thm_1a} to the case where the closed ball is replaced by an arbitrary nonempty compact convex subset of~$\R^n$ containing~$0$ in its interior.

\begin{theorem}\label{thm_3.1}
	Let $K$ be a nonempty compact convex subset of $\R^n$ with $0\in \sint K$ and let $f:K\to\R^n$ be a nonvanishing and continuous function. Then, one has
	\begin{equation}\label{inq5}
		\sup_{x\in K}\|f(x)-x\| \geq r+ \inf_{x\in K} \|f(x)\|,
	\end{equation}
	where 
	\begin{equation}\label{def_r}
		r:= \sup \, \{\rho >0 \mid \bar B_{\mathbb R^n}(0,\rho)\subset K\}.
	\end{equation}
\end{theorem}
  \begin{proof}
  Suppose for the sake of contradiction that~\eqref{inq5} fails. Then, 
  $$\sup_{x\in K}\|f(x)-x\| < r+ \inf_{x\in K} \|f(x)\|.$$
This implies that $$\|f(x)-x\| < r+ \|f(x)\|$$	for all $x\in \bar B_{\mathbb R^n}(0,r)$.	Squaring both sides of this inequality gives
$$\|f(x)\|^2- 2\langle f(x),x\rangle +\|x\|^2 < r^2+2r\|f(x)\| + \|f(x)\|^2$$
for all $x\in \bar B_{\mathbb R^n}(0,r)$. Thus,
\begin{equation}\label{inq4}
	- \langle f(x),x\rangle  < r\|f(x)\| \quad\mbox{\rm for any}\ 	x\in \partial\bar B_{\mathbb R^n}(0,r).
\end{equation}
Applying Theorem~\ref{hs_thm} to the variational inequality associated with the function $f$ and the compact convex set $\bar B_{\mathbb R^n}(0,r)$, we can find a vector $\ox \in \bar B_{\mathbb R^n}(0,r)$ such that
\begin{equation*}\label{VI_f_r}
	\langle f(\bar x), y-\bar x\rangle \geq 0\quad \mbox{\rm for all}\ y\in \bar B_{\mathbb R^n}(0,r).
\end{equation*}
Then, using the nonvanishing property of~$f$ and arguing similarly as in the proof of Theorem~\ref{thm_1a}, we can infer that $\ox\in\partial \bar B_{\mathbb R^n}(0,r)$ and $f(\ox)=-\lambda \ox$ for some $\lambda >0.$ Therefore, $$-\langle f(\ox ),\ox \rangle = \lambda r^2 =r\| f(\ox)\|.$$ This contradicts~\eqref{inq4}. 

We have thus established the inequality~\eqref{inq5}, where the constant $r>0$ is given by~\eqref{def_r}.
$\hfill\Box$
 \end{proof}

 \subsection{Finding sharp estimates}
For any compact convex set $K$ and continuous function $f$ that satisfy the hypotheses of Theorem~\ref{thm_3.1}, the inequality  
\begin{equation}\label{inq6}
	\sup_{x\in K}\|f(x)-x\| \geq \alpha + \inf_{x\in K} \|f(x)\|
\end{equation}
holds for all $\alpha \leq r$, where $r$ is defined in~\eqref{def_r}. This fact follows from~\eqref{inq5}. It is of interest to determine the largest possible value of $\alpha$ for which~\eqref{inq6} remains valid. The next proposition shows that for a given nonempty compact convex set $K$, the value $\alpha$ is bounded above by $\sup\limits_{x\in K} \|x\|$.

  \begin{proposition}\label{prop_4}
	Let $K$ be a nonempty compact convex subset of $\R^n$ with $0\in \sint K$. Suppose that there exists  $\alpha >0$ such that  for any  nonvanishing and continuous function $f:K\to\R^n$, the inequality~\eqref{inq6} holds. Then, we have $\alpha \leq r_1$, where 
\begin{equation}\label{def_r1}
r_1 :=\sup\limits_{x\in K} \|x\|
\end{equation}
  \end{proposition}
  \begin{proof}
  	Suppose to the contrary that  $\alpha > r_1$. Fix any vector $x'\in \R^n$ with $\|x'\| = \alpha$. Then, by~\eqref{def_r1} we have $x'\notin K$. Define $f(x)= x-x'$ for all $x\in K$. Clearly, $f$ is continuous on $K$ and $f(x)\neq 0$ for all $x\in K$. Hence, the compactness of $K$ implies that $\inf\limits_{x\in K} \|f(x)\|=\min\limits_{x\in K}\|f(x)\|>0$. Therefore, we get
  	  	$$	\sup_{x\in K}\|f(x)-x\| = \sup_{x\in K}\|x-x'-x\| =\|x'\|  =\alpha <  \alpha+ \inf_{x\in K} \|f(x)\|.$$
 This contradicts to~\eqref{inq6}. Thus, one has $\alpha \leq r_1$.
 $\hfill\Box$
  \end{proof}
  
As a result of Theorem~\ref{thm_3.1} and Proposition~\ref{prop_4}, the largest possible value of $\alpha$ must belong to the interval $[r,r_1]$, where $r$ and $r_1$ are defined respectively by~\eqref{def_r} and~\eqref{def_r1}.

\begin{proposition}\label{prop_5}
	Let $K$ be a nonempty compact convex subset of $\R^n$ with $0\in \sint K$. Let $\alpha \in[r,r_1]$ be such that the inequality~\eqref{inq6} holds for any nonvanishing and continuous function $f:K\to\R^n$. Then, one must have $\alpha =r$.
\end{proposition}
\begin{proof}
Suppose, to the contrary, that $\alpha > r$. Let $x' \in\partial K$ be such that $\min\limits_{x\in \partial K}\|x\|=\|x'\|$. Then, from~\eqref{def_r} it follows that $\|x'\|=r$. 

Take any $\varepsilon>0$ and define $f(x)= x - (1+\varepsilon)x'$ for all $x\in K$. We see that $f$ is continuous on~$K$. The function $f$ is nonvanishing on $K$. Indeed, if there is some $\ox\in K$ such that $f(\ox)=0$, then $\ox=(1+\varepsilon)x'$. This gives
$$x'= \dfrac{1}{1+\varepsilon}\ox= \dfrac{1}{1+\varepsilon}\ox + \dfrac{\varepsilon}{1+\varepsilon}0.$$ 
Since $\varepsilon >0$, it follows that $x'\in (0,\ox)\subset \sint K$, which contradicts to the fact that $x'\in \partial K$. Thus, $f(x)\neq 0$ for all $x\in K$.

Since $f$ is nonvanishing and continuous on~$K$, the inequality~\eqref{inq6} is fulfilled by our assumption on~$\alpha$. We have
$$\sup_{x\in K}\|f(x)-x\| =\sup_{x\in K}\|x-(1+\varepsilon)x'-x\| =(1+\varepsilon)\|x'\| =(1+\varepsilon)r. $$
Combining this with~\eqref{inq6} yields
$$(1+\varepsilon)r\geq \alpha + \inf_{x\in K} \|f(x)\| > \alpha.$$
Hence, $(1+\varepsilon)r  >\alpha$ for all $\varepsilon>0$. Letting $\varepsilon\to 0$, we get $r\geq \alpha$, which leads to a contradiction.
$\hfill\Box$
\end{proof}

Combining Propositions~\ref{prop_4} and \ref{prop_5}  with Theorem~\ref{thm_3.1} gives us the next theorem.

\begin{theorem}\label{thm_1}
		Let $K$ be a nonempty compact convex subset of $\R^n$ with $0\in \sint K$. Then, the inequality~\eqref{inq6} holds for any nonvanishing and continuous function $f:K\to\R^n$ if and only if $\alpha \leq r$.
\end{theorem}

\subsection{The effect of changing the norm}

Until now, we have focused on the maximum displacements of nonvanishing continuous vector-valued functions when $\R^n$ is equipped with the Euclidean norm. Since all norms in a finite-dimensional vector space are equivalent, one can wonder how the estimates like the one in~\eqref{inq5} may change when considering a different norm instead of the Euclidean norm.

Let  $\|\cdot\|_*$ be an arbitrary norm in~$\R^n$. Then, there exist positive constants~$\theta_1$ and~$\theta_2$ such that
\begin{equation}\label{eqv_norm}
\theta_1 \| x\|_* \leq \|x\| \leq \theta_2\|x\|_*
\end{equation}  for all $x\in \R^n$. By the homogeneous property of norms, the inequalities in~\eqref{eqv_norm} hold for all $x\in \R^n$ if and only if they hold for all $x\in \partial\bar B_{\mathbb R^n}(0,1)$. Thus, by setting
\begin{equation}\label{theta*_1_2}
	\theta^*_1= \min\limits_{x\in \partial \bar B_{\mathbb R^n}(0,1)} \dfrac{1}{\|x\|_*}\quad\mbox{\rm and}\quad \theta^*_2= \max\limits_{x\in \partial \bar B_{\mathbb R^n}(0,1)} \dfrac{1}{\|x\|_*},
\end{equation}
we see that  the inequalities in~\eqref{eqv_norm} holds for all $x\in \R^n$ if and only if $\theta_1\leq\theta^*_1$ and $\theta_2\geq\theta^*_2$. 

The counterpart of estimate~\eqref{inq5} in the norm $\|\cdot\|_*$ can be expressed as follows.

\begin{theorem}\label{thm_6}
	Let $K$ be a nonempty compact convex subset of $\R^n$ with $0\in \sint K$ and let $f:K\to\R^n$ be a nonvanishing continuous function. Then, one has
	\begin{equation}\label{inq7}
		\sup_{x\in K}\|f(x)-x\|_* \geq r_*+ \dfrac{\theta^*_1}{\theta^*_2}\inf_{x\in K} \|f(x)\|_*
	\end{equation}
	with $r_*:= \dfrac{r}{\theta^*_2}$, where $r$ is defined in~\eqref{def_r}.
\end{theorem}
\begin{proof}
By~\eqref{theta*_1_2} and the remark following it, we have
$$\|f(x)-x\| \leq \theta^*_2 \|f(x)-x\|_*\quad\mbox{\rm for all}\ x\in \R^n.$$
Thus, $$\sup_{x\in K}\|f(x)-x\| \leq \theta^*_2\, \sup_{x\in K}\|f(x)-x\|_*.$$ Similarly, by~\eqref{theta*_1_2} and the subsequent remark, $$\theta^*_1  \inf_{x\in K} \|f(x)\|_*\leq  \inf_{x\in K} \|f(x)\|.$$ Therefore, by using~\eqref{inq5} we obtain
$$ \theta^*_2 \sup_{x\in K}\|f(x)-x\|_*\geq \sup_{x\in K}\|f(x)-x\| \geq r + \inf_{x\in K} \|f(x)\| \geq r+ \theta^*_1  \inf_{x\in K} \|f(x)\|_*.$$
This implies that $$\sup_{x\in K}\|f(x)-x\|_*\geq \dfrac{r}{\theta^*_2} +\dfrac{\theta^*_1}{\theta^*_2} \inf_{x\in K} \|f(x)\|_*.$$ Thus,~\eqref{inq7} is valid.
$\hfill\Box$
\end{proof}

From Theorem~\ref{thm_1} and the definitions of $\theta^*_1$ and $\theta^*_2$ it follows that~\eqref{inq7} is the best lower estimate for the maximum displacements of nonvanishing continuous vector-valued functions defined on $K$, if $\R^n$ is equipped with the norm $\|\cdot\|_*$.

\begin{corollary}\label{col_1}
	Let~ $\|\cdot\|_*$ be a norm in~$\R^n$ and $K$ be a nonempty compact convex subset of $\R^n$ with $0\in \sint K$. Then, there exists a constant $\nu>0$ depending only on the norm $\|\cdot\|_*$ such that the inequality 
	\begin{equation}\label{inq7-nu}
		\sup_{x\in K}\|f(x)-x\|_* \geq \nu \left(r+ \inf_{x\in K} \|f(x)\|_*\right),
	\end{equation}
	where $r$ is defined by~\eqref{def_r}, holds for any nonvanishing continuous function $f:K\to\R^n$.
\end{corollary}
\begin{proof}
	Set $\nu:= \min \left\{\dfrac{1}{\theta^*_2} ,\dfrac{\theta^*_1}{\theta^*_2} \right\}$, where $\theta^*_1$ and $\theta^*_2$ are defined by~\eqref{theta*_1_2}. Given  any nonvanishing continuous function $f:K\to\R^n$, we observe that
	$$ r.\dfrac{1}{\theta^*_2}+ \dfrac{\theta^*_1}{\theta^*_2}\inf_{x\in K} \|f(x)\|_*\geq \left(\min \left\{\dfrac{1}{\theta^*_2} ,\dfrac{\theta^*_1}{\theta^*_2}\right\}\right)\left(r+ \inf_{x\in K} \|f(x)\|_*\right).$$
	So, applying Theorem~\ref{thm_6}, we get~\eqref{inq7-nu} from~\eqref{inq7}.
	$\hfill\Box$
\end{proof}

\begin{corollary}\label{col_2}
 Suppose that $X$ is an $n$-dimensional normed space equipped with a norm $\|\cdot\|_*$ and  $g: X\to \R^n$ is a linear bijective mapping. Then, there exists a number $\nu>0$ satisfying the following property: For any $r>0$, one can find a nonempty convex compact subset $K_r$ of $X$ with $0\in\sint K_r$ and with $\|x\|_*\leq r$ for all $x\in K_r$, such that for any nonvanishing continuous function $f: K_r\to X$ it holds
\begin{equation}\label{eqn_col2}
\sup_{x\in K_r}\|f(x)-x\|_* \geq \nu \left(r+ \inf_{x\in K_r} \|f(x)\|_*\right).
\end{equation}
\end{corollary}
\begin{proof}
 Since $g: X\to \R^n$ is a linear bijective mapping, by~\cite[Theorem~1.21]{Rudin_1991} (see also~\cite[Lemma~2.5]{LY_2020}) we can assert that~$g$ is a homeomorphism. Put $r_2 = \dfrac{r \bar{\theta}}{\|g^{-1}\|_* }$, where $\|g^{-1}\|_*:=\sup\limits_{\|y\|=1}\|g^{-1}(y)\|_*$ and $\bar{\theta}:= \min\limits_{x\in \partial \bar B_{\mathbb R^n}(0,1)} \dfrac{1}{\|x\|_*}$. Let $K_r=g^{-1}\left(\bar B_{\mathbb R^n}(0,r_2)\right)$. As $\bar B_{\mathbb R^n}(0,r_2)$ is a compact convex subset of $\R^n$, $K_r$ is a nonempty compact convex set in $X$ with $0\in \sint K_r$. We have $\|y\|_*\leq r$ for every $y\in K_r$.  Indeed, for any $y\in K_r$, we can find $x\in \bar B_{\mathbb R^n}(0,r_2)$ with $y=g^{-1}(x)$. Therefore,
 \begin{equation*}
 	\begin{array}{ll}
 		\|y\|_*=\|g^{-1}(x)\|_*\leq \|g^{-1}\|_*\| x\|_*\leq \dfrac{\|g^{-1}\|_*\| x\|}{\bar{\theta}}\leq \dfrac{r_2\|g^{-1}\|_*}{\bar{\theta}}= r.
 	\end{array}
 \end{equation*}
 Define a function $\|\cdot\|_\# :X\to \R$ by setting $\|y\|_\# =\| g(y)\|$ for all $y\in X.$  As $g$ is a linear bijective mapping, $\|\cdot\|_\#$ is a norm on $X$. Clearly, $\|g^{-1}(x)\|_\#=\|x\|$ for all $x\in \R^n$. Let $f: K_r\to X$ be a nonvanishing continuous function. Then, the mapping $g\circ f\circ g^{-1}:\bar B_{\mathbb R^n}(0,r_2)\to\R^n$ is continuous. In addition, the nonvanishing property of $f$ together with the linearity of $g$ and $g^{-1}$ imply that  $(g\circ f\circ g^{-1}) (x)\neq 0$ for every $x\in \bar B_{\mathbb R^n}(0,r_2)$. Hence, applying Theorem~\ref{thm_3.1} for the nonempty compact convex set $B_{\mathbb R^n}(0,r_2)$ and the function $g\circ f\circ g^{-1}$ yields
\begin{equation}\label{eq_a_prop2}
\sup_{x\in\bar B_{\mathbb R^n}(0,r_2)}\|(g\circ f\circ g^{-1})(x)-x\| \geq r_2+ \inf_{x\in\bar B_{\mathbb R^n}(0,r_2)} \|(g\circ f\circ g^{-1})(x)\|.
\end{equation}
On the other hand, invoking the linearity of~$g$ once again, one has
$$\begin{array}{rcl}
\sup\limits_{x\in \bar B_{\mathbb R^n}(0,r_2)}\|(g\circ f\circ g^{-1})(x)-x\| & = & \sup\limits_{x\in \bar B_{\mathbb R^n}(0,r_2)}\|g^{-1}\Big((g\circ f\circ g^{-1})(x)-x\Big)\|_\#\\ 
 & = & \sup\limits_{x\in \bar B_{\mathbb R^n}(0,r_2)}\|( f\circ g^{-1})(x)-g^{-1}(x)\|_\#\\ & = & \sup\limits_{y \in K_r}\|f(y)-y\|_\#
\end{array}$$
and
$$\inf_{x \in \bar B_{\mathbb R^n}(0,r_2)}\|(g\circ f\circ g^{-1})(x)\|= \inf_{x \in \bar B_{\mathbb R^n}(0,r_2)}\|(f\circ g^{-1})(x)\|_\#= \inf_{y\in K_r}\|f(y)\|_\#.$$
Combining this with~\eqref{eq_a_prop2} gives
\begin{equation}\label{sharp} \sup_{y \in K_r}\|f(y)-y\|_\# \geq r_2+  \inf_{y\in K_r}\|f(y)\|_\#.\end{equation}  Observe that $\|\cdot\|_\#$ depends only on the function $g$. Then, using~\eqref{sharp} and some arguments similar to those in the proofs of Theorem~\ref{thm_6} and Corollary~\ref{col_1} with the Euclidean norm being replaced by $\|\cdot\|_\#$, we can find $\eta>0$ depending only on the function $g$ such that 
$$ \sup_{x\in K_r}\|f(x)-x\|_* \geq \eta \left(r_2+ \inf_{x\in K_r} \|f(x)\|_*\right)=\eta \left(\dfrac{r\bar{\theta}}{\|g^{-1}\|_*}+ \inf_{x\in K_r} \|f(x)\|_*\right).$$
Setting $\nu= \eta\min\left\{\dfrac{\bar{\theta}}{\|g^{-1}\|_*, 1}\right\}$, we obtain~\eqref{eqn_col2} from the last inequality.
$\hfill\Box$
\end{proof}

\section{Other Five Questions and Their Solutions}\label{Sect-4}

In the preceding two sections, we have solved the first question among the six ones proposed by Professor Biagio Ricceri~\cite{Ricceri_2024} and provided several related results. In this section, we will study the remaining five questions. For each question, we will try either to give a complete solution or to provide a near-to-complete partial solution. In doing so, we will establish meaningful properties of nonvanishing continuous vector-valued functions.

\subsection{Question~2}

The next question reminds us of the concept of eigenvalue of linear operators. To be more precise, here we are concerned with approximate eigenvalues in a sense.

\smallskip
\noindent\textbf{Question 2:} \textit{Can one find a continuous function $f:\bar B_{\mathbb R^n}(0,1)\to \mathbb R^n$ with $f(x)\neq 0$ for every $x\in \bar B_{\mathbb R^n}(0,1)$ such that there is a $\mu>0$ satisfying 
	\begin{equation}\label{eigenvalue}
		\|f(x)-\mu x\|\leq \mu\quad \mbox{\rm for all}\ x\in \bar B_{\mathbb R^n}(0,1)?
\end{equation}}

\begin{proposition}\label{prop_7}
There is no function $f$ fulfilling the required conditions.
\end{proposition}
\begin{proof}
 Suppose on contrary that there is a nonvanishing continuous function $f:\bar B_{\mathbb R^n}(0,1)\to \mathbb R^n$ such that there is a $\mu>0$ satisfying~\eqref{eigenvalue}. Then, applied to the variational inequality defined by the function $f$ and the nonempty compact convex constraint set $\bar B_{\mathbb R^n}(0,1)$, the Hartman-Stampacchia Theorem (see Theorem~\ref{hs_thm}) gives a vector $\bar x\in \bar B_{\mathbb R^n}(0,1)$ such that
\begin{equation}\label{VI}
	\langle f(\bar x), y-\bar x\rangle \geq 0\quad \mbox{\rm for all}\ y\in \bar B_{\mathbb R^n}(0,1).
\end{equation}
The case $\bar x\in {\rm int}\,\bar B_{\mathbb R^n}(0,1)$ is excluded, because if $\bar x\in {\rm int}\,\bar B_{\mathbb R^n}(0,1)$ then~\eqref{VI} yields $f(\bar x)=0$. So, $\bar x\in \partial \bar B_{\mathbb R^n}(0,1)$. Combining this with~\eqref{VI}, we can assert that $f(\bar x)=-\lambda \bar x$ for some $\lambda\ge 0$. The situation $\lambda=0$ cannot occur, because $f(x)\neq 0$ for every $x\in \bar B_{\mathbb R^n}(0,1)$. Hence, $f(\bar x)=-\lambda \bar x$ with $\lambda>0$. But, the latter inequality is in conflict with~\eqref{eigenvalue} since we have

\begin{equation*}
	\begin{array}{rcl}
		\mu\geq \|f(\bar x)-\mu \bar x\| & = & \|\mu \bar x-f(\bar x)\|\\
		& = & \sup\limits_{u\in \bar B_{\mathbb R^n}(0,1)}\, \langle \mu \bar x-f(\bar x), u\rangle\\
		&\geq & \langle \mu \bar x-f(\bar x), \bar x\rangle\\
		& = & \langle \mu \bar x+\lambda \bar x, \bar x\rangle\\
		& = & \langle (\mu +\lambda) \bar x, \bar x\rangle\\
		& = & \mu +\lambda.
	\end{array}
\end{equation*} The proof is complete.
$\hfill\Box$
\end{proof}

Interestingly, Proposition~\ref{prop_7} can be proved via Theorem~\ref{thm_1a}. In fact, the following refined answer for Question~2 is valid.

\begin{theorem}\label{thm_1a-appl}
	For any nonvanishing continuous function $f:\bar B_{\mathbb R^n}(0,1)\to\R^n$ and any number $\mu>0$, one has 
	\begin{equation}\label{inq1-appl}
			\sup_{x \in \bar B_{\mathbb R^n}(0,1)}\|f(x)-\mu x\|\geq \mu +\inf\limits_{x \in \bar B_{\mathbb R^n}(0,1)} \|f(x)\|.
	\end{equation} Hence, there is no function $f$ satisfying the conditions stated in Question~2.
\end{theorem}
\begin{proof} Let $f:\bar B_{\mathbb R^n}(0,1)\to\R^n$ be an arbitrary  nonvanishing continuous function. Given any number $\mu>0$, we put $g(x)=\frac{1}{\mu}f(x)$ for all $x\in \bar B_{\mathbb R^n}(0,1)$. Clearly, the function $g:\bar B_{\mathbb R^n}(0,1)\to\R^n$ is nonvanishing and continuous. Therefore, by Theorem~\ref{thm_1a} we have $d_g \geq \gamma_g$. Hence,
in agreement with the notations defined in~\eqref{df} and~\eqref{gf}, we get
\begin{equation*}
	\sup_{x \in \bar B_{\mathbb R^n}(0,1)} \|g(x) - x\| \geq 1 + \inf_{x \in \bar B_{\mathbb R^n}(0,1)} \|g(x)\|.
\end{equation*} Since $g(x)=\frac{1}{\mu}f(x)$ for every $x\in \bar B_{\mathbb R^n}(0,1)$, multiplying both sides of the last inequality by $\mu$ yields~\eqref{inq1-appl}. The first assertion of the theorem has been proved.

For any nonvanishing continuous function $f:\bar B_{\mathbb R^n}(0,1)\to\R^n$ and any number $\mu>0$, by the Weierstrass theorem (see, e.g., \cite[Theorem 1.59]{bmn2022})  we find that the number $\alpha:=\inf\limits_{x \in \bar B_{\mathbb R^n}(0,1)} \|f(x)\|$ is positive and there is a vector $\bar x$ with 
$$\sup_{x \in \bar B_{\mathbb R^n}(0,1)}\|f(x)-\mu x\|=\|f(\bar x)-\mu\bar x\|.$$ So, from~\eqref{inq1-appl} it follows that
	$$\|f(\bar x)-\mu\bar x\|\geq \mu +\alpha>\mu.$$ This shows that the property~\eqref{eigenvalue} cannot hold.
$\hfill\Box$
\end{proof}

As an illustration for Theorem~\ref{thm_1a-appl}, let us consider a simple example.

\begin{example} {\rm Setting $f(x)=x^2+1 $ for all $x\in [-1,1]$, we have a nonvanishing continuous function $f:\bar B_{\mathbb R^1}(0,1)\to\R^1$. It is easily shown that  $\sup\limits_{x \in \bar B_{\mathbb R^1}(0,1)}|f(x)-\mu x|=2+\mu$ for every $\mu>0$, and the supremum is attained at $x=-1$. Since $\mu+\inf\limits_{x \in \bar B_{\mathbb R^n}(0,1)} \|f(x)\|=\mu+1$,~the inequality~\eqref{inq1-appl} is strict for all $\mu>0$.}
\end{example}

\subsection{Question~3}

It is well known (see, e.g.,~\cite[Proposition~7.1]{Deimling_1985}) that the closed unit ball in an infinite-dimensional normed space is noncompact in the normed topology. Hence, the results from the previous sections cannot be fully generalized to infinite-dimensional normed spaces, making the next question both interesting and challenging.

\medskip
\noindent{\bf Question 3:} {\it Can we find a finite-dimensional linear subspace $F$ of a Hilbert space ${\cal H}$ and a continuous function $\varphi: {\cal H}\to {\cal H}$ such that
\begin{itemize}
\item[\rm (a)] $\sup\limits_{x\in {\cal H}} \|\varphi(x)-x\| <+\infty$;
\item[\rm (b)] For each $x\in {\cal H}$, there exists $y\in F$ such that $\langle \varphi(x),y\rangle \neq 0$?
\end{itemize}}

To solve this question, we need the following auxiliary result, whose proof is based on Corollary~\ref{col_2}.

\begin{lemma}\label{lem_1}
Let $F$ be a finite-dimensional linear subspace of a Hilbert space~${\cal H}$. Then, there is no continuous operator $\varphi:F\to F$ such that
\begin{itemize}
	\item[\rm (i)] $\sup\limits_{x\in F} \|\varphi(x)-x\| <+\infty$;
	\item[\rm (ii)] For each $x\in F$, there exists $y\in F$ such that $\langle \varphi(x),y\rangle \neq 0$.
\end{itemize}
\end{lemma}
\begin{proof}
On the contrary, suppose that there is a continuous operator $\varphi:F\to F$ satisfying the conditions~(i) and~(ii). It follows from~(ii) that $\varphi(x)\neq 0$ for all $x\in F$. By Corollary~\ref{col_2}, there is a constant $\nu>0$ such that, for any $r>0$, one can find a nonempty compact convex set $K_r$ in $F$ with $0\in\sint K_r$, $\|x\|\leq r$ for all $x\in K_{r}$, and 
$$\sup_{x\in K_r}\|\varphi(x)-x\| \geq \nu \left(r+ \inf_{x\in K_r} \|\varphi(x)\|\right).$$
Since $\sup\limits_{x\in F}\|\varphi(x)-x\|\geq \sup\limits_{x\in K_r}\|\varphi(x)-x\|$, this yields
$$\sup_{x\in F}\|\varphi(x)-x\|\geq  \nu \left(r+ \inf_{x\in K_r} \|\varphi(x)\|\right)\geq \nu r.$$
Letting $r\to\infty$, we get $\sup\limits_{x\in F}\|\varphi(x)-x\| =+\infty$, which contradicts the condition~(i).
$\hfill\Box$
\end{proof}

The next theorem gives an answer in the negative to Question~3.

\begin{theorem}\label{thm_8}
Suppose that $F\subset {\cal H}$ is a finite-dimensional linear subspace. If  $\varphi: {\cal H}\to {\cal H}$ is such a continuous function that for each $x\in {\cal H}$ there exists $y\in F$ satisfying $\langle \varphi(x),y\rangle \neq 0$, then $\sup\limits_{x\in {\cal H}} \|\varphi(x)-x\|=+\infty$.
\end{theorem}
\begin{proof}
We have ${\cal H}= F \oplus F^\perp$, where $F^\perp$ is the orthogonal complement of~$F$. Define the continuous linear operators $P_1:{\cal H}\to F$ and $ P_2:{\cal H}\to F^\perp$ as follows. For each $x\in {\cal H}$ with $x=x_1+x_2$, where $x_1\in F$, $x_2\in F^\perp$, put $P_1(x)=x_1$ and $P_2(x)=x_2$. For every $x\in F$, since $P_1\big(\varphi(x)\big)-x\in F$, we have
\begin{align*}
\|\varphi(x)-x\|&= \|P_1\big(\varphi(x)\big)+P_2\big(\varphi(x)\big)-x\| \\
&= \sqrt{\|P_1\big(\varphi(x)\big)-x\|^2 +\|P_2\big(\varphi(x)\big)\|^2 + 2\langle P_1\big(\varphi(x)\big)-x, P_2\big(\varphi(x)\big)\rangle}\\
&= \sqrt{\|P_1\big(\varphi(x)\big)-x\|^2 +\|P_2\big(\varphi(x)\big)\|^2 }\\
&\geq \|P_1\big(\varphi(x)\big)-x\|.
\end{align*}
Therefore, $\sup\limits_{x\in F}\|\varphi(x)-x\| \geq\sup\limits_{x\in F} \|P_1\big(\varphi(x)\big)-x\|.$ This yields 
\begin{equation}\label{estimates}
\sup_{x\in {\cal H}}\|\varphi(x)-x\| \geq \sup\limits_{x\in F}\|\varphi(x)-x\|  \geq \sup_{x\in F} \|P_1\big(\varphi(x)\big)-x\|.
\end{equation}
By the assumption on $\varphi$, for each $x\in{\cal H}$, there exists $y\in F$ such that $\langle \varphi(x),y\rangle\neq 0$. As
$$\langle \varphi(x),y\rangle = \langle P_1\big(\varphi(x)\big)+P_2\big(\varphi(x)\big),y\rangle  =\langle P_1\big(\varphi(x)\big), y\rangle,$$ we get $\langle P_1\big(\varphi(x)\big), y\rangle \neq 0$. Hence, the continuous function $P_1\circ \big(\varphi{\big|}_{F}\big):F\to F$ has the following property: \textit{For each $x\in F$, there exists $y\in F$ such that $\big\langle \left(P_1\circ \big(\varphi{\big|}_{F}\big)\right)(x),y\big\rangle \neq 0$.} So, thanks to Lemma~\ref{lem_1} we know that 
$$\sup_{x\in F} \|P_1\big(\varphi(x)\big)-x\|=\sup_{x\in F} \|\left(P_1\circ \big(\varphi{\big|}_{F}\big)\right)(x)-x\|=+\infty.$$ 
Combining this with~\eqref{estimates} gives $\sup\limits_{x\in {\cal H}} \|\varphi(x)-x\|=+\infty$, and completes the proof.
$\hfill\Box$ 
\end{proof}

\begin{remark} {\rm The requirement~(b) in Question~3 can be termed as a \textit{nonvanishing property of $\varphi: {\cal H}\to {\cal H}$ with respect to the linear subspace $F$}. In connection with the concepts discussed in Sections~\ref{Sect-2} and~\ref{Sect-3}, the number $\sup\limits_{x\in {\cal H}} \|\varphi(x)-x\|$ can be called the \textit{maximum displacement} of $\varphi$. Then, the result in Theorem~\ref{thm_8} reads as follows: \textit{The maximum displacement of a continuous function $\varphi: {\cal H}\to {\cal H}$, which is nonvanishing w.r.t. a finite-dimensional linear subspace $F\subset {\cal H}$, is infinity.}}
\end{remark}

\begin{remark} {\rm The finite-dimensionality of the linear subspace $F$ is an essential assumption of Theorem~\ref{thm_8}. Indeed, if~${\cal H}$ is infinite-dimensional, then the maximum displacement of a continuous function $\varphi: {\cal H}\to {\cal H}$, which is nonvanishing w.r.t. ${\cal H}$, can be finite (see Theorem~\ref{thm_1-6} below).}
\end{remark}

\subsection{Question~4}

Apart from using the normed spaces setting, the next question is very different from Question~3 by its contents.

\medskip
\noindent {\bf Question 4:} {\it Let $E$ be an infinite-dimensional real normed space. Can we find a continuous operator $\psi: E\to E$ and finitely many $\varphi_1,\dots,\varphi_n\in E^*$	such that
	\begin{itemize}
\item[\rm (a1)] $\sup\limits_{x\in E} \|\psi(x)-x\| <+\infty$;
\item[\rm (b1)] For each $x\in E$ there exists $i\in\{1,\dots,n\}$ such that $\varphi_i(\psi(x))\neq 0$?
	\end{itemize}}

To answer this question, let us first prove the following lemma.

\begin{lemma}\label{lem_2}
Let $E$ be an infinite-dimensional real normed space and $L_0$ be a finite-dimensional subspace of $E$. Then, there is no continuous operator $\psi_0: L_0\to L_0$ such that
\begin{itemize}
	\item[\rm (i)] $\sup\limits_{x\in L_0} \|\psi_0(x)-x\| <+\infty$;
	\item[\rm (ii)] $\psi_0(x)\neq 0$ for all $x\in L_0$.
\end{itemize}
\end{lemma}
\begin{proof}
Suppose to the contrary that there exists a continuous operator $\psi_0: L_0\to L_0$ such that~(i) and~(ii) hold. By Corollary~\ref{col_2} we can find a constant $\nu>0$ such that for any $r>0$ there is a nonempty compact convex set $K_r$ in~$L_0$ with $0\in K_r$ and $\|x\|\leq r$ for all $x\in K_r$ such that
$$\sup_{x\in K_r}\|\psi_0(x)-x\| \geq \nu \left(r+ \inf_{x\in K_r} \|\psi_0(x)\|\right).$$
Thus, we get
	$$\sup_{x\in L_0}\|\psi_0(x)-x\|\geq\sup_{x\in K_r}\|\psi_0(x)-x\| \geq \nu \left(r+ \inf_{x\in K_r} \|\psi_0(x)\|\right)\geq \nu r.$$ Letting $r\to+\infty$, we deduce that $\sup\limits_{x\in L_0}\|\psi_0(x)-x\| =+\infty$, which contradicts~(i).
	$\hfill\Box$
\end{proof}

The following theorem provides us with a solution to Question~4.

\begin{theorem}\label{prop_9}
Suppose that $\varphi_1,\dots,\varphi_n$ is a finite system of continuous linear functionals on an infinite-dimensional real normed space $E$. If $\psi:E\to E$ is a continuous operator with $\sup\limits_{x\in E} \|\psi(x)-x\| <+\infty$, i.e., $\psi$ has a finite maximum displacement, then there must exist a vector $\bar x\in X$ such that $\varphi_i(\psi(\bar x))=0$ for all  $i\in\{1,\dots,n\}$. In other words,  the conditions~{\rm (a1)} and~{\rm (b1)} in Question~4 cannot hold simultaneously. 
\end{theorem}

\begin{proof}
Suppose to the contrary that there exist  $\varphi_1,\dots,\varphi_n\in E^*$ and a continuous operator $\psi: E\to E$ satisfying the conditions~{\rm (a1)} and~{\rm (b1)}. Consider the linear mapping $A:E\to\R^n$ defined by 
$$A(x)=(\varphi_1(x),\dots,\varphi_n(x))\quad\mbox{\rm for}\ x\in E.$$
Then, $M:=A(E)$ is a subspace of $\R^n$. Put $m=\dim M$. It follows from~(b1) that $m \geq 1$. So, there exist $v^1,\dots,v^m\in M$ such that $M =\mbox{\rm span}\; \{v^1,\dots,v^m\}$ and the vectors $v^1,\dots,v^m$ are linearly independent. Take $u^1,\dots, u^m\in E$ such that $A(u^j)=v^j$ for $j\in\{1,\dots,m\}$. Clearly, $u^1,\dots, u^m$ are linearly independent. Define the sets
\begin{equation*}
\begin{array}{lll}
L_0& = &\mbox{\rm span}\; \{u^1,\dots,u^m\},\\
L^\perp& = &\big\{x\in E\mid \varphi_i(x)=0\ \mbox{\rm for all}\; i\in\{1,\dots,n\}\big\}.
\end{array}
\end{equation*}

We see that $E= L_0\oplus L^\perp$. Indeed, take any $x\in E$. Then, there exists a unique vector $(\xi_1,\dots,\xi_m)\in\R^m$ such that 
$$A(x)=\sum\limits_{j=1}^m \xi_j v^j= \sum\limits_{j=1}^m \xi_j A(u^j).$$
For $x_0:= \sum\limits_{j=1}^m \xi_j u^j$, one has $x_0\in L_0$ and 
$$A(x_0)= A\left(\sum\limits_{j=1}^m \xi_j u^j\right)= \sum\limits_{j=1}^m \xi_j A(u^j)= A(x).$$
Therefore, $A(x-x_0)=0$, which implies that $\varphi_i(x-x_0)=0$ for all $i\in\{i,\dots, n\}$. Thus, $x-x_0\in L^\perp$. Since $x\in E$ is chosen arbitrarily, $E=L_0+L^\perp$. In addition, if $x\in (L_0\cap L^\perp)$, then $A(x)=0$ and there exist $(\xi_1,\dots,\xi_m)\in R^m$ such that $x =\sum\limits_{j=1}^m \xi_j u^j$. Hence,
$$0 = A(x)  = A\left(\sum\limits_{j=1}^m \xi_j u^j\right)= \sum\limits_{j=1}^m \xi_j A(u^j) =\sum\limits_{j=1}^m \xi_j v^j.$$
Combining this with the linear independence of $\{v^1,\dots,v^m\}$ gives $\xi_j=0$ for all $j\in\{1,\dots,m\}$, which implies $x=0$. We have thus proved that $E= L_0\oplus L^\perp$. 

We now show that the restriction $A{\big|}_{L_0}: L_0\to M$ is a bijective mapping. Suppose that $A(x)=A(y)$ for some $x,y\in L_0$. Then, there exist some vectors $(\xi_1,\dots,\xi_m), (\mu_1,\dots,\mu_m)\in \R^m$ such that
\begin{equation*}
x=\sum\limits_{j=1}^m \xi_j u^j\quad\mbox{\rm and}\quad y  =\sum\limits_{j=1}^m \mu_j u^j.
\end{equation*}
So, we have
\begin{align*}
0=A(x)-A(y) &  = A\left(\sum\limits_{j=1}^m \xi_j u^j\right)- A\left(\sum\limits_{j=1}^m \mu_j u^j\right)\\ &= \sum\limits_{j=1}^m \xi_j A(u^j) - \sum\limits_{j=1}^m \mu_j A(u^j)\\&= \sum\limits_{j=1}^m (\xi_j-\mu_j) v^j.
\end{align*}
This implies that $\xi_j=\mu_j$ for all $j\in\{1,\dots,m\}$. Or equivalently, $x=y$. Moreover, for any $w\in M$, there is $(\omega_1,\dots,\omega_m)\in \R^m$ such that $w= \sum\limits_{j=1}^m \omega_j v^j$. Thus, $$w= \sum\limits_{j=1}^m \omega_j A(u^j)= A\left(\sum\limits_{j=1}^m \omega_j u^j\right).$$
So, $w= A(x)$ where $x= \sum\limits_{j=1}^m \omega_j u^j\in L_0$. Therefore, $A{\big|}_{L_0}:L_0\to M$ is a one-to-one mapping. Thus, by~\cite[Theorem~1.21]{Rudin_1991} or~\cite[Lemma~2.5]{LY_2020}, we can infer that $A{\big|}_{L_0}$ is a homeomorphism.

\smallskip
\noindent {\sc Claim 2.} {\it For any $\rho \in\R$, the set $S(\rho):= \{z\in L_0\mid \max\limits_{i\in\{1,\dots,n\}} \varphi_i(z) \leq \rho\}$ is bounded.}

Indeed, taking any $\rho\in\R$ for any $z\in S(\rho)$, we have
$$\|A(z)\| = \left\|\big(\varphi_1(z),\dots, \varphi_n(z)\big)\right\|=\sqrt{\sum\limits_{i=1}^n [\varphi_i(z)]^2}\leq \sqrt{n\rho^2}.$$
Since the linear mapping $A{\big|}_{L_0}$ is a homeomorphism, $A{\big|}_{L_0}^{-1}$ is a bounded linear operator, and hence, $S(\rho)$ is bounded.

Let $\psi_0:E\to L_0$ and  $\psi_1:E\to L^\perp$ be defined by setting $\psi_0(x)=x_0$ and $\psi_1(x)=x_1$ for every $x=x_0+x_1$ with $x_0\in L_0$ and $x_1\in L^\perp$. It follows from condition~(a1) that there is $\beta \geq 0$ such that $\sup\limits_{x\in E} \|\psi(x)-x\| =\beta$. Thus, $ \|\psi(x)-x\| \leq \beta$ for all $x\in E$. Put $\gamma = \max\limits_{i\in\{1,\dots,n\}} \|\varphi_i\|$. For any $x\in E$ and for any $i\in\{1,\dots,n\}$, one has
$$\varphi_i \left(\psi(x)-x\right) \leq \|\varphi_i \|\|\psi(x)-x\|\leq \gamma\beta.$$
Furthermore, for every $i\in\{1,\dots,n\}$,
$$\begin{array}{rcl}
\varphi_i \left(\psi(x)-x\right) = \varphi_i \left(\psi_0(x)+\psi_1(x)-x\right) & = &\varphi_i \left(\psi_0(x)-x\right)+\varphi_i\left(\psi_1(x)\right)\\ & = & \varphi_i \left(\psi_0(x)-x\right).
\end{array}
$$
Thus, $\psi_0(x)-x \in S(\gamma\beta)$ for all $x\in E$. Hence, by Claim~2, there exists $\sigma\in \R$ such that $\|\psi_0(x)-x\| \leq \sigma$ for all $x\in E$. So, 
$$\|\psi_0(x)\|\leq \sigma + \|x\|.$$
This gives $\|\psi_0(x)\|\leq \sigma+ 1$ for all $x\in E$ with $\|x\|\leq 1$. Note that the mapping $\psi_0{\big|}_{L_0}: L_0\to L_0$ is linear and bounded, and hence, it is continuous. We also have
\begin{equation}\label{inq8}
\sup\limits_{x\in E} \|\psi_0(x)-x\| \leq \sigma <+\infty.
\end{equation}
On the other hand, noting that for every $x\in E$ and for every $i\in\{1,\dots,n\}$, 
$$\varphi_i\left(\psi(x)\right) = \varphi_i\left(\psi_0(x)+\psi_1(x)\right)= \varphi_i\left(\psi_0(x)\right)+ \varphi_i\left(\psi_1(x)\right)= \varphi_i\left(\psi_0(x)\right),$$
we deduce from~(b1) that for every $x\in E$, there exists $i\in\{1,\dots,n\}$ such that $\varphi_i\left(\psi_0(x)\right)\neq 0$. It follows that $\psi_0(x)\neq 0$ for all $x\in E$. Combining this with~\eqref{inq8}, we see that the continuous function  $\psi_0{\big|}_{L_0}:L_0\to L_0$ satisfies the conditions~(i) and~(ii) in Lemma~\ref{lem_2}. We have arrived at a contradiction. So, there is no $\varphi_1,\dots,\varphi_n\in E^*$ and continuous operator $\psi:E\to E$ such that the conditions~{\rm (a1)} and~{\rm (b1)} hold simultaneously. 
$\hfill\Box$
\end{proof}

\subsection{Question~5}

Unlike the above Question~4, where the continuous function $\psi$ maps~$E$ to~$E$, the next question asks about some properties of a continuous function $\psi$ mapping~$E$ to its dual space $E^*$. These properties are stated in connection with a linear and continuous operator $\varphi$, that also maps $E$ to~$E^*$. In this context, the nonvanishing feature and the maximum displacement of continuous vector-valued functions, which have been studied intensively until now, are formulated in more sophisticated and quite new forms.

\medskip
\noindent \textbf{Question 5:} \textit{Let $E$ be a normed space, $F\subset E$ a linear subspace with ${\rm dim}(F)<\infty$, $\psi: E\rightarrow E^*$ a continuous mapping, and $\varphi: E\rightarrow E^*$ a linear and continuous operator such that 
\begin{itemize}
		\item[{\rm (a2)}] For every $y\in E$ there exists $x\in F$ satisfying $\psi(y)(x)\neq 0$; 
		\item[{\rm (b2)}] For every $x\in F\setminus \{0\}$ there exists $y\in E$ satisfying $\varphi(y)(x)\neq 0$.
\end{itemize} Then, does the equality 
	\begin{equation}\label{supremum}
		\sup\limits_{y\in E}\|\varphi(y)-\psi(y)\|_{E^*}=+\infty
	\end{equation} hold true?} 

\medskip
The next theorem gives an answer in the affirmative to Question~5.

\begin{theorem}\label{thm_4.4}
	If $E$ is a normed space, $F\subset E$ is a linear subspace with ${\rm dim}(F)<\infty$, $\psi: E\rightarrow E^*$ is a continuous mapping, and $\varphi: E\rightarrow E^*$ is a linear and continuous operator such that both conditions {\rm (a2)} and {\rm (b2)} are satisfied, then the equality~\eqref{supremum} is valid.
\end{theorem}

\noindent{\bf Proof.} If $E$ is a trivial normed space, i.e., $E=\{0\}$, then one has $F=\{0\}$ and $E^*=\{0\}$. Hence,~(a2) fails to hold. Thus, the assertion of the theorem is valid.

Consider the case $E$ has nonzero vectors. Let ${\rm dim}(F)=m$. Condition~(a2) implies that $m\geq 1$. Indeed, taking any $\bar y\in E$, we have $\psi(\bar y)\in E^*$. By~(a2), we can find $\bar x\in F$ such that $\varphi(\bar y)(\bar x)\neq 0$. This forces $\bar x\neq 0$. Hence, ${\rm dim}(F)\geq 1$. 

Clearly, the properties {\rm (a2)} and {\rm (b2)} are still valid if the norm in $F$ is replaced by another equivalent norm. So, we can equip $F$ with a scalar product and the corresponding Euclidean norm, denoted by $\|.\|_2$. Via the scalar product, $F^*$ can be identified with $F$. Let $\theta$ be a positive constant such that $\|x^*\|_{F^*}\geq\theta \|x^*\|_2$ for all $x^*\in F^*$.

\smallskip
\noindent {\sc Claim~1.} \textit{The set $L:=\big\{\varphi(y){\big|}_{F}\mid y\in E\big\}$, where $\varphi(y){\big|}_{F}$ denotes the restriction of $\varphi(y)$ on~$F$, coincides with $F^*$.}

Indeed, by the linearity of the operator $\varphi: E\rightarrow E^*$, $L$ is a linear subspace of the dual space~$F^*$, which has been identified with $F$. Besides, for every $x^*\in F^*$ and every $x\in F$, one has $x^*(x)=\langle x^*,x\rangle$. If $L$ is a proper linear subspace of $F$, then there exists $\bar x$ in $F\setminus L$ such that $\langle z,\bar x\rangle =0$ for all $z\in L$. It follows that $$0=\langle \varphi(y){\big|}_{F},\bar x\rangle=\varphi(y){\big|}_{F}(\bar x)= \varphi(y)(\bar x)$$ for all $y\in E$, contrary to~(b2). We have thus proved that $L=F^*$. 

\smallskip
Since the linear operator $\Phi: E\to L$, where $\Phi(y):=\varphi(y){\big|}_{F}$ for all $y\in E$, is continuous and surjective, Claim~1 allows us to find a linear subspace $M$ of $E$ with ${\rm dim}\, M=m$ such that $E={\rm ker}(\Phi)\oplus M$ (see~\cite[Definition~4.20]{Rudin_1991} for the definition of \textit{direct sum} and ~\cite[Lemma~4.21]{Rudin_1991} for the existence of a \textit{complement} of ${\rm ker}(\Phi)$ in $E$). By the construction of $M$, $\Phi{\big|}_{M}:M\to L$ is a linear homeomorphism. Hence, the inverse map $\Phi{\big|}_{M}^{-1}:L\to M$ is a continuous linear operator. 

We have
\begin{eqnarray}\label{transformations}
	\begin{array}{rcl} \sup\limits_{y\in E}\|\varphi(y)-\psi(y)\|_{E^*} &=& \sup\limits_{y\in E}\left(\sup\limits_{v\in \bar B_E(0,1)}\big(\varphi(y)-\psi(y)\big)(v)\right)\\
		&\geq& \sup\limits_{y\in E}\left(\sup\limits_{v\in \bar B_F(0,1)}\big(\varphi(y)-\psi(y)\big)(v)\right)\\
		&=& \sup\limits_{y\in E}\left(\sup\limits_{v\in \bar B_F(0,1)}\big(\varphi(y){\big|}_{F}-\psi(y){\big|}_{F}\big)(v)\right)\\
		&=&	\sup\limits_{y\in E}\|\varphi(y){\big|}_{F}-\psi(y){\big|}_{F}\|_{F^*}\\
		&\geq&	\sup\limits_{y\in M}\|\varphi(y){\big|}_{F}-\psi(y){\big|}_{F}\|_{F^*}\\
		&\geq&	\theta\,\sup\limits_{z\in F}\|z-\psi\Big(\Phi{\big|}_{M}^{-1}(z)\Big){\big|}_{F}\|_2,
	\end{array}
\end{eqnarray} where the change of variable $z=\varphi(y){\big|}_{F}$, which yields $y=\Phi{\big|}_{M}^{-1}(z)$, has been done in the last transformation. The mapping $\psi\Big(\Phi{\big|}_{M}^{-1}(\cdot)\Big){\big|}_{F}:F\to F$, which is defined by setting $\psi\Big(\Phi{\big|}_{M}^{-1}(\cdot)\Big){\big|}_{F}(z)=\Big[\psi\Big( \Phi{\big|}_{M}^{-1}(z)\big)\Big]{\big|}_{F}$ for all $z\in F$,  is continuous by the continuity of $\psi$. In addition, for every $z\in F$, by using condition~(a2) for $y:=\Phi{\big|}_{M}^{-1}(z)$ we can find $x\in F$ such that $\psi(y)(x)\neq 0$. Since the latter can be rewritten as $\langle \psi\Big(\Phi{\big|}_{M}^{-1}(z)\Big){\big|}_{F},x\rangle\neq 0$, we can apply Lemma~\ref{lem_1} for $f:=\psi\Big(\Phi{\big|}_{M}^{-1}(\cdot)\Big){\big|}_{F}$ to obtain $$\sup\limits_{z\in F}\|z-\psi\Big(\Phi{\big|}_{M}^{-1}(z)\Big){\big|}_{F}\|_2=+\infty.$$ From this and~\eqref{transformations} we can infer that the equality~\eqref{supremum} is valid. 

The proof is complete. $\hfill\Box$

\subsection{Question~6}

It is of interest to know what happen if the condition ${\rm dim}(F)<\infty$ in Question~5 and in Theorem~\ref{thm_4.4} is violated. Probably, this is one of the strong motivations for the following question of Professor Ricceri~\cite{Ricceri_2024}.

\medskip
\noindent \textbf{Question 6:} \textit{Let $E$ be a normed space and $F\subset E$ a linear subspace with ${\rm dim}(F)=\infty$. Does there exist a continuous mapping  $\psi: E\rightarrow E^*$ and a linear and continuous operator $\varphi: E\rightarrow E^*$ satisfying the conditions \begin{itemize}
		\item[{\rm (a3)}] For every $y\in E$ there exists $x\in F$ satisfying $\psi(y)(x)\neq 0$; 
		\item[{\rm (b3)}] For every $x\in F\setminus \{0\}$ there exists $y\in E$ satisfying $\varphi(y)(x)\neq 0$
	\end{itemize} such that the inequality 
	\begin{equation}\label{supremum_a1-6}
		\sup\limits_{y\in E}\|\varphi(y)-\psi(y)\|_{E^*}<+\infty
	\end{equation} holds true?} 

\medskip
 To solve the above question in the affirmative in one special case, we need the next auxiliary result on nonvanishing continuous operators having finite maximum displacements. Note that the continuous function $\psi$ in Question~6 maps $E$ to the dual space $E^*$, while the function $\psi$ in the forthcoming theorem maps $E$ to $E$.

\begin{theorem}\label{thm_1-6} 
	Let $E$ be an infinite-dimensional normed space. Then, there exists a nonvanishing continuous operator $\psi:E\to E$ such that
	\begin{equation}\label{supremum_a2}
		\sup\limits_{y\in E}\|y-\psi(y)\|\leq 1.
	\end{equation}
\end{theorem}

\noindent{\bf Proof.} Given an infinite-dimensional normed space $E$, we can find a continuous mapping $f:\bar B_E(0,1)\to \bar B_E(0,1)$, which has no fixed point (see, e.g.,~\cite[Remarks~8.7, p.~66]{Deimling_1985}). Define a map $\psi:E\to E$ by the formula
\begin{equation}\label{psi}
	\psi(x)=\begin{cases}
		x-f(x) & {\rm if}\ x\in \bar B_E(0,1)\\
		x-f\left(\dfrac{x}{\|x\|}\right) & {\rm if}\ x\notin \bar B_E(0,1).
	\end{cases}
\end{equation}

{\sc Claim~1}. \textit{$\psi$ is nonvanishing.}

\smallskip
Indeed, if $x\in \bar B_E(0,1)$, then $\psi(x)=x-f(x)\neq 0$ because the mapping $f:\bar B_E(0,1)\to \bar B_E(0,1)$ has no fixed point. If $x\notin \bar B_E(0,1)$, then $\|x\|>1$. Therefore, if $\psi(x)=x-f\left(\dfrac{x}{\|x\|}\right)=0$, then one has
$$1<\|x\|=\|f\left(\dfrac{x}{\|x\|}\right)\|\leq 1,$$ which is impossible.

\smallskip
{\sc Claim~2}. \textit{$\psi$ is continuous.}

\smallskip
Since the identity mapping is continuous, by~\eqref{psi} it is clear that the claim will be proved if we can show that the function
\begin{equation}\label{f_tilde}
	\widetilde f(x):=\begin{cases}
		f(x) & {\rm if}\ x\in \bar B_E(0,1)\\
		f\left(\dfrac{x}{\|x\|}\right) & {\rm if}\ x\notin \bar B_E(0,1)
	\end{cases}
\end{equation} is continuous on $E$. Take any $\bar x\in E$ and suppose that $V\subset E$ is an open set containing~$\widetilde f(\bar x)$. There is an open neighborhood $U$ of $\bar x$ such that
\begin{equation}\label{neighborhood_U}
	\widetilde f(x)\in V\quad {\rm for\ all}\ x\in U. 
\end{equation} 
Indeed, if $\|\bar x\|<1$, then there exists an open set $U_0\subset B_E(0,1)$ with $\bar x\in U_0$. As $\widetilde f(x)=f(x)$ for all $x\in U_0$, by the continuity of $f$ we can find an open set  $U\subset U_0$ such that $\bar x\in U$ and $f(x)\in V$ for all $x\in U$. Then, from~\eqref{f_tilde} it follows that~\eqref{neighborhood_U} holds with this $U$. 

If $\|\bar x\|=1$, then by the continuity of $f:\bar B_E(0,1)\to \bar B_E(0,1)$ we can find an open set $U_1$ containing $\bar x$ such that
$f\big(U_1\cap \bar B_E(0,1)\big)\subset V\cap\bar B_E(0,1).$ Hence, from~\eqref{f_tilde} it follows that
\begin{equation}\label{U1}
	\widetilde f\big(U_1\cap \bar B_E(0,1)\big)\subset V. 
\end{equation} Clearly, the map $g:E\setminus\{0\}\to \bar B_E(0,1)$, where $g(x):=\dfrac{x}{\|x\|}$ for every $x\neq 0$, is continuous. So, the map $f\circ g:E\setminus\{0\}\to \bar B_E(0,1)$ is continuous. We observe for any $x\in E\setminus\bar B_E(0,1)$ that
	$$(f\circ g)(x)=f(g(x))=f\left(\dfrac{x}{\|x\|}\right)=\widetilde f(x).$$ Note that $(f\circ g)(\bar x)=f(\bar x)=\widetilde f(\bar x)\in V$. So, by the continuity of the map $f\circ g:E\setminus\{0\}\to \bar B_E(0,1)$ we can find an open set $U_2$ containing $\bar x$ such that $(f\circ g)\big(U_2\big)\subset V$. Therefore,
	\begin{equation}\label{U2a}
		\widetilde f\Big(U_2\cap \big(E\setminus \bar B_E(0,1)\big)\Big)=(f\circ g)\Big(U_2\cap \big(E\setminus \bar B_E(0,1)\big)\Big)\subset V.
	\end{equation}
	Setting $U=U_1\cap U_2$, we deduce from~\eqref{U1} and~\eqref{U2a} that~\eqref{neighborhood_U} is valid.

If $\|\bar x\|>1$, then by the continuity of $f\circ g$, where  $g:E\setminus\{0\}\to \bar B_E(0,1)$ is defined by setting $g(x)=\dfrac{x}{\|x\|}$ for every $x\neq 0$, and the fact that
$$(f\circ g)(x)=\widetilde f(x)\quad {\rm for\ every}\ x\in E\setminus\bar B_E(0,1)$$ we can find an open neighborhood $U$ of $\bar x$ with $U\subset E\setminus\bar B_E(0,1)$ such that $(f\circ g)(U)\subset V$. This implies ~\eqref{neighborhood_U} and completes the proof of the continuity of $\widetilde f$.

\smallskip
{\sc Claim~3}. \textit{The inequality~\eqref{supremum_a2} holds.}

\smallskip
The desired inequality follows directly from~\eqref{psi}.

\smallskip
Summing up all the above, we conclude that the operator $\psi:E\to E$ constructed by~\eqref{psi} is nonvanishing,  continuous, and it satisfies the inequality~\eqref{supremum_a2}. $\hfill\Box$

\medskip
The next theorem gives an answer in the affirmative to Question~6 for one class of normed spaces. The question remains open if $E$ is not a Hilbert space.

\begin{theorem}\label{thm_2-6}
	For any infinite-dimensional Hilbert space $E$ and for any a linear subspace $F\subset E$ with ${\rm dim}(F)=\infty$, there exist a continuous mapping  $\psi: E\rightarrow E^*$ and a continuous linear operator $\varphi: E\rightarrow E^*$ satisfying the conditions~{\rm (a3)} and~{\rm (b3)} in Question~6 such that the inequality~\eqref{supremum_a1-6} holds true.
\end{theorem}

\noindent{\bf Proof.} As $E$ is a Hilbert space, we can identify $E^*$ with $E$. Besides, for the given infinite-dimensional linear subspace $F\subset E$, we have $E=F\oplus F^\perp$, where $F^\perp$ denotes the orthogonal complement of $F$ in $E$. By $\pi$ we denote the orthogonal projection from $E$ onto $F$. Then, for every $x\in E$, it holds that $x_1:=\pi (x)$ belongs to $F$, $x_2:=x-\pi(x)$ belongs to $F^\perp$, and $x=x_1+x_2$. Since ${\rm dim}(F)=\infty$, according to~\cite[Remarks~8.7, p.~66]{Deimling_1985}, we can find a continuous mapping $g:\bar B_F(0,1)\to \bar B_F(0,1)$ having no fixed point. Define a map $\psi_F:F\to F$ by the formula
\begin{equation}\label{psi_F}
	\psi_F(u)=\begin{cases}
		u-g(u) & {\rm if}\ u\in \bar B_F(0,1)\\
		u-g\left(\dfrac{u}{\|u\|}\right) & {\rm if}\ u\in F\setminus\bar B_F(0,1).
	\end{cases}
\end{equation} The arguments given in the proof of Theorem~\ref{thm_1-6} show that $\psi_F:F\to F$ is a  nonvanishing continuous operator. Now, we can construct an operator $\psi:E\to E^*$, where $E^*=E$, by setting
\begin{equation}\label{psi_E}
	\psi(x)=\big(x-\pi(x)\big)+\psi_F(\pi(x)) 
\end{equation} for all $x\in E$. Clearly, $\psi:E\to E$ is a continuous operator. 

For every $y\in E$, since $\psi_F(\pi(y))\in F\setminus\{0\}$ and  ${\rm dim}(F)=\infty$, there exists $x\in F$ satisfying $\langle \psi_F(\pi(y)),x\rangle\neq 0$. From~\eqref{psi_E} it follows that 
$$\langle \psi(y),x\rangle= \langle \big(y-\pi(y)\big)+\psi_F(\pi(y)),x\rangle= \langle \psi_F(\pi(y)),x\rangle,$$ where the second equality is valid because $y-\pi(y)\in F^\perp$. Hence, we get $\langle \psi(y),x\rangle\neq 0$. Thus, the condition~(a3) is satisfied. 

Setting $\varphi(x)=x$ for all $x\in E$, we see at once that $\varphi:E\to E^*$, where $E^*=E$, is a linear and continuous operator. To verify the condition~(b3), take any  $x\in F\setminus \{0\}$ and choose $y=x$. Then one has $\varphi(y)(x)=\langle y,x\rangle =\langle x,x\rangle\neq 0$. 

Furthermore, for every $y\in E$, it holds that 
\begin{equation}\label{norms_1}\begin{array}{rcl}
		\|\varphi(y)-\psi(y)\|_{E^*}=\|y-\psi(y)\|_{E} &=&\|y-\big[(y-\pi(y))+\psi_F(\pi(y))\big]\|_{E}\\
		&=&\|\pi(y)-\psi_F(\pi(y))\|_{E}.
	\end{array}
\end{equation}
In addition, by~\eqref{psi_F} we have
\begin{equation}\label{norms_2}
	\|\pi(y)-\psi_F(\pi(y))\|_{E}=\begin{cases}
		\|g(\pi(y))\| & {\rm if}\ \pi(y)\in \bar B_F(0,1)\\
		\|g\left(\dfrac{\pi(y)}{\|\pi(y)\|}\right)\| & {\rm if}\ \pi(y)\in F\setminus\bar B_F(0,1).
	\end{cases}
\end{equation} Combining~\eqref{norms_1} with~\eqref{norms_2} yields $\|\varphi(y)-\psi(y)\|_{E^*}\leq 1$. So, one gets
\begin{equation*}\label{supremum_a1}
	\sup\limits_{y\in E}\|\varphi(y)-\psi(y)\|_{E^*}\leq 1.
\end{equation*}
We have thus shown that the inequality~\eqref{supremum_a1-6} holds for the chosen operators~$\psi$ and~$\varphi$.
$\hfill\Box$

\section{Concluding Remarks}\label{Sect-5}

We have studied six open questions related to the maximum displacements and generalized maximum displacements of nonvanishing continuous vector-valued functions in finite-dimensional and infinite-dimensional normed spaces. Our main results are the following:

- Sharp lower estimates for the maximum displacements (Theorems~\ref{thm_1a},~\ref {thm_1}, and~\ref{thm_6});

- A refined solution to a question on a type of positive approximate eigenvalues of nonvanishing continuous vector-valued functions (Theorem~\ref{thm_1a-appl});

- Sufficient conditions for the infinite  maximum displacements of nonvanishing continuous vector-valued functions on Hilbert spaces (Theorem~\ref{thm_8});

- Necessary conditions for nonvanishing continuous vector-valued functions on infinite-dimensional normed spaces to have finite  maximum displacements (Theorem~\ref{prop_9});

- Sufficient conditions for a nonvanishing continuous vector-valued function on a normed space to possess infinite generalized maximum displacements of (Theorem~\ref{thm_4.4});

- The existence of nonvanishing continuous vector-valued functions on an arbitrary infinite-dimensional normed space with finite maximum displacements (Theorem~\ref{thm_1-6}); 

- The existence of a pair of continuous vector-valued functions on an arbitrary infinite-dimensional Hilbert space satisfying certain conditions, which has finite generalized maximum displacement (Theorem~\ref{thm_2-6}).

\medskip
These new results and the proofs show how the Hartman-Stampacchia Theorem on the solution existence of variational inequalities can be useful for the investigations of rather difficult questions in nonlinear functional analysis.

\begin{acknowledgements}
	This research was supported by the project NCXS02.01/24-25 of the Vietnam Academy of Science and Technology. The authors are indebted to Professor Biagio Ricceri for suggesting the open questions and for warm hospitality at Catania University, Italy, during their one-month stay in October 2024.
\end{acknowledgements}

\end{document}